\documentclass[10pt, a4paper]{article}

\usepackage{color,url,latexsym,amssymb,amsthm,amsmath,amsxtra,amsfonts}
\usepackage{graphicx}
\usepackage{amsfonts}
\usepackage{amsbsy}
\usepackage{amssymb}
\usepackage{amsmath}
\usepackage{amsthm}
\theoremstyle{plain}
\newtheorem{theorem}{Theorem}[section]

\newtheorem{corollary}{Corollary}[section]
\newtheorem{proposition}{Proposition}[section]
\theoremstyle{definition}
\newtheorem{definition}{Definition}[section]
\newtheorem{remark}{Remark}[section]
\newtheorem{example}{Example}[section]

\begin{document}

\title{On a  Lichnerowicz type cohomology attached to a function}
\author{Cristian Ida}
\date{}
\maketitle

\begin{abstract}
In this paper we define a new cohomology of a smooth manifold called Lichnerowicz type cohomology attached to a function. Firstly, we study some basic properties of this cohomology as: a de Rham type isomorphism, dependence on the function, singular forms, relative cohomology, Mayer-Vietoris sequence, homotopy invariance and next, a regular case is considered. The notions are introduced using techniques from the study of two
cohomologies of a smooth manifold: the Lichnerowicz cohomology and the cohomology attached to a function. 
\end{abstract}

\medskip

\begin{flushleft}
\strut \textbf{2010 Mathematics Subject Classification}: 58A10, 58A12.

\end{flushleft}

\maketitle

\section{Introduction}
\setcounter{equation}{0}

Let us consider a $n$-dimensional smooth manifold $M$ and $\theta$ a closed $1$-form on $M$. Denote by $\Omega^r(M)$ the set of all $r$-differential forms on $M$ and consider the twisted operator $d_{\theta}:\Omega^r(M)\rightarrow\Omega^{r+1}(M)$ defined by $d_{\theta}=d-\theta\wedge$, where $d$ is the usual exterior derivative. Since $d\theta=0$, we easily obtain $d_{\theta}^2=0$. The differential complex $(\Omega^\bullet(M),d_\theta)$ is called the \textit{Lichnerowicz complex} of $M$; its cohomology groups $H^{\bullet}_{\theta}(M)$ are called the \textit{Lichnerowicz cohomology groups} of $M$. This is the classical Lichnerowicz cohomology,  motivated by Lichnerowicz's work \cite{L} or Lichnerowicz-Jacobi cohomology on Jacobi and locally conformal symplectic geometry manifolds, see \cite{Ba, L-L-M}. It is also known in literature as Morse-Novikov cohomology and  plays an important role when studying the geometry, topology and Morse theory of the underlying manifold $M$, see for instance \cite{O-S, O-V1, O-V2, Pa}. We also notice that Vaisman \cite{Va5}, studied it under the name of "adapted cohomology" in the context of locally conformally K\"{a}hler manifolds. Locally, the Lichnerowicz cohomology complex becames the de Rham complex after a change $\varphi\mapsto e^{-f}\varphi$ with $f$ a smooth function which satisfies $df=\theta$, that is $d_{\theta}$ is the unique differential in $\Omega^{\bullet}(M)$ which makes the multiplication by the smooth function $e^{-f}$ an isomorphism of cochain complexes $e^{-f}:(\Omega^{\bullet}(M),d_{\theta})\rightarrow(\Omega^{\bullet}(M),d)$. In the case when $\theta$ is exact, there is a related differential $\delta_{t,f}=d+tdf\wedge$ ($t\in\mathbb{R}$ and $f$ is a Morse function on $M$) introduced in \cite{Wi} by Witten in order to obtain an analytic proof of the Morse inequalities. Also, the cohomology of the deformed Witten differential is very useful in the study of the topology of the fiber $f^{-1}(c)$, see \cite{D-S, F-S}.

Other important applications of the Lichnerowicz cohomology appear when studying locally (globally) conformal structures: (co)symplectic, (co)K\"{a}hler, Jacobi, Dirac etc., and in many cases it is an invariant at conformal changes. For more about this cohomology see for instance \cite{Ba, H-R, L-L-M, Va5, Wa}. Also, such a cohomology can be generalized in the context of Jacobi algebroid, that is a pair $(A,\theta)$ where $A=(A,[\cdot,\cdot]_A,\rho_A)$ is a Lie algebroid over a manifold $M$ and $\theta\in\Gamma(A^*)$ is a $1$-cocycle, i.e $d_A\theta=0$, see \cite{C-N-C, I-M}. Then the twisted cohomological operator is given by $d_{A,\theta}=d_A-\theta\wedge$, and the Lichnerowicz cohomology of $(M,\theta)$ is just the Jacobi Lie algebroid cohomology of usual Lie algebroid $(TM,[\cdot,\cdot], Id)$ with the $1$-cocycle $\theta$.

On the other hand, in \cite{Mo2},  Monnier gave the definition and basic properties of a new cohomology of a manifold, called \textit{cohomology attached to a function}. The definition is the following: If $f$ is a smooth function on a smooth manifold $M$,  then we can define the linear operator $d_f:\Omega^r(M)\rightarrow\Omega^{r+1}(M)$ by
\begin{displaymath}
d_f\varphi=fd\varphi-rdf\wedge\varphi\,,\,\forall\,\varphi\in\Omega^r(M).
\end{displaymath}
It is easy to see that $d_f^2=0$, and so, we have a differential complex $(\Omega^\bullet(M),d_f)$ which is called the \textit{differential complex attached to the function $f$} of $M$; its cohomology groups $H^{\bullet}_{f}(M)$ are called the \textit{cohomology groups attached to the function $f$} of $M$. This cohomology was considered for the first time in \cite{Mo1} in the context of Poisson geometry, and more generally, Nambu-Poisson geometry.

The main difference between the operators $d_f$ and $d_{\theta}$ is given by the fact that $d_f$ is an antiderivation, i.e. 
\begin{equation}
d_f(\varphi\wedge\psi)=d_f\varphi\wedge\psi+(-1)^{\deg\varphi}\varphi\wedge d_f\psi, \,\forall\,\varphi,\psi\in\Omega^{\bullet}(M),
\label{I3}
\end{equation}
while $d_{\theta}$ is not an antiderivation, and it satisfies
\begin{equation}
d_{\theta}(\varphi\wedge\psi)=d\varphi\wedge\psi+(-1)^{\deg\varphi}\varphi\wedge d_{\theta}\psi, \,\forall\,\varphi,\psi\in\Omega^{\bullet}(M).
\label{I4}
\end{equation}
Moreover,  according with \cite{Mo2}, the $d_f$--cohomology is just the Lie algebroid cohomology associated with a certain Lie algebroid structure $([\cdot,\cdot]_f,\rho_f)$ on $TM$, where the anchor is defined by $\rho_f(X)=fX$ and the Lie bracket is defined by $[X,Y]_f=(1/f)[fX,fY]$. Thus, $d_f$-cohomology is a Lie algebroid cohomology, while $d_\theta$-cohomology is a Jacobi algebroid cohomology.

In this paper we introduce a generalization of the usual Lichnerowicz cohomology starting from the cohomology attached to a function. We observe that if $\theta$ is a closed one form on an arbitrary smooth manifold $M$ and $f\in C^\infty(M)$ then $d_f(f\theta)=0$, that is $(TM,f\theta)$ is a Jacobi algebroid where $(TM,[\cdot,\cdot]_f,\rho_f)$ is the above Lie algebroid. Then, it is natural to consider the cohomological differential operator $d_{f,\theta}:\Omega^r(M)\rightarrow\Omega^{r+1}(M)$ defined by 
\begin{equation}
d_{f,\theta}\varphi=d_{f}\varphi-f\theta\wedge\varphi\,,\,\,\varphi\in\Omega^r(M),
\label{II1}
\end{equation} 
which satisfies $d_{f,\theta}\circ d_{f,\theta}=0$. Thus, we obtain the differential complex $(\Omega^\bullet(M),d_{f,\theta})$ which is called  the \textit{Lichnerowicz type complex attached to the function $f$} of $M$ and its cohomology groups $H^{\bullet}_{f,\theta}(M)$ are called the \textit{Lichnerowicz type cohomology groups attached to the function $f$} of $M$. This is just the Jacobi algebroid cohomology of $(TM,f\theta)$ with the Lie algebroid structure $([\cdot,\cdot]_f,\rho_f)$ and, using some arguments as in the study of $d_f$-cohomology, the main goal of this paper is to study the main properties of this new cohomology of a smooth manifold. 

The paper is organized as follows: In the second section we make some remarks about our cohomology when the one form $\theta$ is exact and we relate it with globally conformal cosymplectic manifolds and with Lichnerowicz-Poisson cohomology twisted by a Hamiltonian vector field on oriented $2$-dimensional Poisson manifolds. In the third section we  study some basic properties of the Lichnerowicz type cohomology attached to a function. The notions are introduced by combining results from the study of Lichnerowicz cohomology and from cohomology attached to a function, and we show that many properties of the de Lichnerowicz cohomology and the cohomology attached to a function still have their analogues within the our Lichnerowicz type cohomology attached to a function. We notice that certain properties concerning to our cohomology are closely related to Lichnerowicz cohomology, while anothers one are closely related to cohomology attached to a function. Firstly, we prove a de Rham type isomorphism theorem for our cohomology (Theorem 3.2), we prove that this cohomology is isomorphic to Lichnerowicz cohomology of singular forms (Proposition 3.1), we discuss how the cohomology varies when the function $f$ changes (Proposition 3.2) and  how it depends on the class of $\theta$ (Proposition 3.3). In particular, we show that if the function $f$ does not vanish, then the Lichnerowicz type cohomology attached to a function is isomorphic to Lichnerowicz cohomology (Corollary 3.1). Next we study a relative cohomology associated to our cohomology and we will show that it is possible to write a Mayer-Vietoris exact sequence (Theorem 3.3). We also give an appropriate notion of homotopy, but it is an open question whether the cohomology is homotopy invariant in general. In the four section we consider the regular case, i.e., the case where the function $f$ does not have singularities in a neighborhood of $S = f^{-1}(\{0\})$. In a similar manner with the study from \cite{Mo2} concerning to cohomology attached to a function,  we can relate our cohomology with the Lichnerowicz cohomology of $M$ and of $S$ (Theorem 4.1). Also, in this regular case, we obtain a homotopy invariance (Proposition 4.1) and (Proposition 4.2).

\section{Some remarks when $\theta$ is exact}
\setcounter{equation}{0}
In this section we make some remarks about our Lichnerowicz cohomology attached to a function in the case when the $1$-form $\theta$ is exact. 
\begin{remark}
Firstly, we remark that the cohomological operators $d_\theta$ and $d_f$ can be canonically associated to a globally conformal almost cosymplectic manifold as follows. According to \cite{Olz, Va6} the \textit{locally conformal almost cosymplectic} manifolds are defined to be almost contact metric manifolds whose almost contact and fundamental forms $\eta$ and $\Phi$ are related by 
\begin{equation}
\label{lcc}
d\eta=\omega\wedge\eta\,\,{\rm and}\,\,d\Phi=2\omega\wedge\Phi
\end{equation}
for some closed 1-form $\omega$. Moreover, if $\omega$ is exact then their are called \textit{globally conformal almost cosymplectic} manifolds, see \cite{C-M}. From \eqref{lcc} it follows that $d_\omega\eta=0$ and $d_{2\omega}\Phi=0$. Also, using \eqref{I4} we have $d_{2r\omega}\Phi^r=0$ and $d_{(2r+1)\omega}(\eta\wedge\Phi^r)=0$, $r=0,\ldots,m$, $\dim M=2m+1$, that is, a (locally) globally conformal almost cosymplectic manifold $M$ has the following cohomological invariants in the Lichnerowicz cohomology: $[\Phi^r]\in H^{2r}_{2r\omega}(M)$ and $[\eta\wedge\Phi^r]\in H^{2r+1}_{(2r+1)\omega}(M)$, $r=0,\ldots,m$. Now, if $M$ is a globally conformal almost cosymplectic manifold with exact $1$-form $\omega$ given by $\omega=d(\log f)$ for some positive function $f\in C^\infty(M)$, the first condition of \eqref{lcc} is equivalent with $d_f\eta=0$ and the second condition of \eqref{lcc} is equivalent with $d_f\Phi=0$. Also, using \eqref{I3} we have $d_f\Phi^r=0$ and $d_f(\eta\wedge\Phi^r)=0$, $r=0,\ldots,m$. Thus, a globally conformal almost cosymplectic manifold $M$ has another cohomological invariants in the cohomology attached to the function $f$, namely $[\Phi^r]\in H^{2r}_f(M)$ and $[\eta\wedge\Phi^r]\in H^{2r+1}_f(M)$, $r=0,\ldots,m$. 

As $\eta$ is $d_f$-closed it is a $1$-cocycle for the Lie algebroid $(TM,[\cdot,\cdot]_f,\rho_f)$ and this leads to a Jacobi algebroid cohomology of a globally conformal almost cosymplectic manifold $H^\bullet_{f,\eta}(M)$ with the cohomological operator $d_{f,\eta}=d_f-\eta\wedge$.  

Now, using $d_f(\eta\wedge\Phi^r)=0$ we get $d_{f,\eta}(\eta\wedge\Phi^r)=0$. Then, for every $r=0,\ldots,m$ we have another cohomology class $[\eta\wedge\Phi^r]$ in $H^{2r+1}_{f,\eta}(M)$. Moreover, it is easy to see that $\eta\wedge\Phi^r=d_{f,\eta}(-\Phi^r)$, which implies that this cohomology class vanishes. 

Also, let us consider $(\widetilde{\eta}=\alpha\eta, \widetilde{\Phi}=\alpha^2\Phi)$ ($\alpha$ is a positive smooth function on $M$) another globally conformal cosymplectic structure on $M$ with exact Lee form $\widetilde{\omega}=d\log\widetilde{f}$, where $\widetilde{f}=\alpha f$. The structures $(\eta,\Phi)$ and $(\widetilde{\eta},\widetilde{\Phi})$ are usually called \textit{conformally equivalent} globally conformal cosymplectic structures on $M$. Then, a direct computation leads to 
\begin{displaymath}
d_{\widetilde{f},\widetilde{\eta}}\varphi=\alpha^{r+1}d_{f,\eta}\left(\frac{\varphi}{\alpha^r}\right),\,\,\varphi\in\Omega^r(M),
\end{displaymath}
or equivalently, $H^\bullet_{f,\eta}(M)\cong H^\bullet_{\widetilde{f},\widetilde{\eta}}(M)$. Hence, $H^\bullet_{f,\eta}(M)$ is an invariant of the class of conformally equivalent globally conformal cosymplectic structures on $M$.  

\end{remark}

\begin{remark}
Let $(M,\Pi)$ be a $n$-dimensional Poisson manifold, that is $\Pi\in\mathcal{V}^2(M)$ satisfies $[\Pi,\Pi]_{SN}=0$, where $[\cdot,\cdot]_{SN}:\mathcal{V}^k(M)\times\mathcal{V}^l(M)\rightarrow\mathcal{V}^{k+l-1}(M)$ denotes the Schouten-Nijenhuis bracket on multi-vectors fields, see for instance \cite{Va7}. There is a remarkable cohomology associated with a Poisson manifold $(M,\Pi)$, called the \textit{Lichnerowicz-Poisson cohomology} denoted by $H^\bullet_{LP}(M,\Pi)$ and realized as the homology of the differential complex $(\mathcal{V}^\bullet(M),d_\Pi)$, where $d_{\Pi}:\mathcal{V}^k(M)\rightarrow\mathcal{V}^{k+1}(M)$, $d_\Pi(X)=[X,\Pi]_{SN}$. Note that $d_\Pi(\alpha)=X_\alpha$, where $X_\alpha$ is the Hamiltonian vector field of $\alpha\in C^\infty(M)$. In the case when the Poisson manifold $(M,\Pi)$ is $2$-dimensional and oriented, Monnier proves in \cite{Mo1, Mo2} that the LP cohomology is isomorphic with the cohomology attached to a function as follows: Consider $\nu\in\Omega^2(M)$ be a fixed volume form of $(M,\Pi)$ and the smooth function $f:=\imath_\Pi\nu$. Then there exist the maps $\phi^p:\mathcal{V}^p(M)\rightarrow\Omega^p(M)$, $p=0,1,2$ defined by

\begin{equation}
\label{pois1}
\phi^0:C^\infty(M)\rightarrow C^\infty(M)\,,\,\phi^0(\alpha)=\alpha\,,\,\forall\,\alpha\in C^\infty(M)=\mathcal{V}^0(M),
\end{equation}
\begin{equation}
\label{pois2}
\phi^1:\mathcal{V}^1(M)\rightarrow\Omega^1(M)\,,\phi^1(X)=-\imath_X\nu\,,\,\forall\,X\in\mathcal{V}^1(M),
\end{equation}
\begin{equation}
\label{pois3}
\phi^2:\mathcal{V}^2(M)\rightarrow\Omega^2(M)\,,\,\phi^2(Y)=(\imath_Y\nu)\nu\,,\,\forall\,Y\in\mathcal{V}^2(M)
\end{equation}
which yields an isomorphism of differential complexes $\phi^p:(\mathcal{V}^p(M),d_\Pi)\rightarrow(\Omega^p(M),d_f)$, $p=0,1,2$, that is
\begin{equation}
\label{pois4}
d_f(\phi^0(\alpha))=\phi^1(d_\Pi(\alpha)) \,\,{\rm and}\,\, d_f(\phi^1(X))=\phi^2(d_\Pi(X))\,,\,\forall\,\alpha\in C^\infty(M)\,,\,\forall\,X\in\mathcal{V}^1(M),
\end{equation}
and consequently, an isomorphism between $H^\bullet_{LP}(M,\Pi)$ and $H^\bullet_f(M,\Pi)$.

In order to relate our Lichnerowicz cohomology attached to a function with a $2$-dimensional Poisson manifold, we will consider the  Lichnerowicz-Poisson cohomology of $(M,\Pi)$ twisted by the Hamiltonian vector field $X_\alpha$ of a smooth function $\alpha\in C^\infty(M)$, i.e. the homology of the differential complex $(\mathcal{V}^\bullet(M),d_{\Pi,\alpha})$, where
\begin{equation}
\label{pois5}
d_{\Pi,\alpha}(X)=d_\Pi(X)-X_\alpha\wedge X\,,\,X\in\mathcal{V}^p(M),\,p=0,1,2.
\end{equation} 
Note that a such twisted cohomology can be defined in a more general setting replacing the Hamiltonian vector field $X_\alpha$ by a Poisson vector field $Z\in\mathcal{V}^1(M)$, i.e. $d_\Pi(Z)=0$, see \cite{Cic}. Let us denote by $H^\bullet_{LP,\alpha}(M,\Pi)$ the twisted LP cohomology by the Hamiltonian vector field $X_\alpha$. Then we have
\begin{proposition}
Let $(M,\Pi)$ be an oriented $2$-dimensional Poisson manifold and $\alpha\in C^\infty(M)$. Then the cohomologies $H^\bullet_{LP,\alpha}(M,\Pi)$ and $H^\bullet_{f,\theta}(M,\Pi)$ are isomorphic, where $\theta=d\alpha$.
\end{proposition}
\begin{proof}
For $g\in C^\infty(M)$ and $\theta=d\alpha$ we have
\begin{equation}
\label{pois6}
d_{f,\theta}(\phi^0(g))=\phi^1(d_{\Pi,\alpha}(g)).
\end{equation}
Indeed
\begin{eqnarray*}
d_{f,\theta}g&=&d_fg-fgd\alpha =d_fg-gd_f\alpha\\
&=&\phi^1(d_\Pi(g))-g\phi^1(d_\Pi(\alpha))=\phi^1(d_\Pi(g)-gX_\alpha)\\
&=&\phi^1(d_{\Pi,\alpha}(g)),
\end{eqnarray*}
where we have used the first relation from \eqref{pois4} and $d_\Pi(\alpha)=X_\alpha$.

Now, starting from $d\alpha\wedge\nu=0$, for $X\in\mathcal{V}^1(M)$ we have $\imath_X(d\alpha\wedge\nu)=0$ which is equivalent with $(\imath_Xd\alpha)\nu=d\alpha\wedge\imath_X\nu$. Multiplying with $f$ we obtain
\begin{equation}
\label{pois7}
fd\alpha\wedge\imath_X\nu=f(\imath_Xd\alpha)\nu.
\end{equation}
But, $f(\imath_Xd\alpha)=\imath_X(fd\alpha)=\imath_X\phi^1(X_\alpha)=-\imath_X\imath_{X_\alpha}\nu=-\imath_{X_\alpha\wedge X}\nu$, hence the relation \eqref{pois7} becomes
\begin{displaymath}
fd\alpha\wedge\imath_X\nu=-(\imath_{X_\alpha\wedge X}\nu)\nu
\end{displaymath}
which is equivalent with
\begin{equation}
\label{pois8}
-fd\alpha\wedge\phi^1(X)=-\phi^2(X_\alpha\wedge X).
\end{equation}
Using the second relation of \eqref{pois4} we get
\begin{displaymath}
d_f(\phi^1(X))-fd\alpha\wedge\phi^1(X)=\phi^2(d_\Pi(X))-\phi^2(X_\alpha\wedge X)
\end{displaymath}
or
\begin{equation}
\label{pois9}
d_{f,\theta}(\phi^1(X))=\phi^2(d_{\Pi,\alpha}(X)).
\end{equation}
Thus, the isomorphism follows by \eqref{pois6} and \eqref{pois9}.
\end{proof}
\end{remark}
\begin{remark}
Another remark which relate our cohomological operator with some computations in \cite{Mo1, Mo2} is the following: Monnier also introduce another cohomological operator attached to a function defined by $d_{f,p}:\Omega^r(M)\rightarrow\Omega^{r+1}(M)$, where
\begin{displaymath}
d_{f,p}\varphi=fd\varphi-(r-p)df\wedge\varphi\,,\,p\in\mathbb{Z}\,,\,\varphi\in\Omega^r(M),
\end{displaymath}
and the cohomology of the differential complex $(\Omega^\bullet(M),d_{f,p})$ is denoted by $H^{\bullet}_{f,p}(M)$. As well as it is noted in Remark 2.3 from \cite{Mo2}, these cohomologies does not come from a Lie algebroid, but will we see that the cohomology $H^\bullet_{f,p}(M)$ come from the Jacobi algebroid cohomology $(TM,[\cdot,\cdot]_f,\rho_f)$ with the $1$-cocycle $f\theta$, where $\theta=-d(\log f^p)$ ($f$ is a positive smooth function on $M$). Indeed, using  \eqref{II1}, it is easy to see that in this case we have $d_{f,\theta}=d_{f,p}$, hence $H^\bullet_{f,\theta}(M)\cong H^\bullet_{f,p}(M)$ and some computations of these cohomologies can be found in \cite{Mo1, Mo2}. 
\end{remark}

\section{Basic properties of the Lichnerowicz type cohomology attached to a function}
\setcounter{equation}{0}

In this section we study several properties of this new cohomology in relation with some classical properties of Lichnerowicz cohomology and of cohomology attached to a function.

\subsection{A de Rham type theorem}
The spaces $H^\bullet_{f,\theta}(M)$ can be also obtained as the cohomology spaces of $M$ with coefficients in a sheaf, that is the sheaf $\Phi_{\theta}(M)$ of germs of smooth functions $\alpha$ on $M$ which are such that $d_{f,\theta}\alpha=fd_\theta\alpha=0$ (or $d_\theta\alpha=0$). More exactly, using the Poincar\'{e} type Lemma for the Lichnerowicz operator $d_\theta$ (see \cite{Va5}), we obtain a Poincar\'{e} type Lemma for the Lichnerowicz type cohomology attached to a function and, consequently, a de Rham type theorem for this cohomology.
\begin{theorem}
\label{pllf}
Let $f$ be a nonvanishing smooth function on $M$ and $\theta$ a closed one form on $M$. If $\varphi$ is a $d_{f,\theta}$-closed $r$-form defined on a neighborhood $U$  on $M$ and $r\geq 1$, then there exists a  $(r-1)$-form $\psi $ defined on some neighborhood $U^{'}\subset U$ and such that $d_{f,\theta}\psi=\varphi$ on $U^{'}$.
\end{theorem}
\begin{proof}
It is easy to see that for a nonvanishing smooth function $f$ on $M$ and $\varphi\in\Omega^r(M)$ we have
\begin{equation}
\label{x3}
d_{f,\theta}\varphi=f^{r+1}d_\theta\left(\frac{\varphi}{f^r}\right).
\end{equation}
Then, if $\varphi\in\Omega^r(U)$ is $d_{f,\theta}$-closed we obtain that $\varphi/f^r$ is $d_\theta$-closed, and applying the  Poincar\'{e} Lemma for the Lichnerowicz operator $d_\theta$ (see Proposition 3.1 from \cite{Va5}) it follows that there is  $\psi^\prime\in\Omega^{r-1}(U^\prime)$, $U^\prime\subset U$ such that $\varphi/f^r=d_\theta\psi^\prime$. Now, if we take $\psi=f^r\psi^\prime$ and we use \eqref{x3} we get $\varphi=d_{f,\theta}\psi$ on $U^\prime$.
\end{proof}
Let  $\Phi^{r}(M)$ the sheaf of germs of $r$-forms on $M$, $\Phi_{\theta}(M)$ the sheaf of germs of smooth functions $\alpha$ on $M$ which are such that $d_{f,\theta}\alpha=fd_\theta\alpha=0$ and $i:\Phi_{\theta}(M)\rightarrow \Phi^{0}(M)$ the natural inclusion. The sheafs $\Phi^{r}(M)$ are fine and taking into account Theorem \ref{pllf} it results that the following sequence of sheafs
\begin{displaymath}
0\longrightarrow \Phi_{\theta}(M)\stackrel{i}{\longrightarrow}\Phi^{0}(M)\stackrel{d_{f,\theta}}{\longrightarrow}\Phi^{1}(M)\stackrel{d_{f,\theta}}{\longrightarrow }\ldots\stackrel{d_{f,\theta}}{\longrightarrow }\Phi^{r}(M)\stackrel{d_{f,\theta}}{\longrightarrow}\ldots 
\end{displaymath}
is a fine resolution of $\Phi_{\theta}(M)$ and we denote by $H^r(M,\Phi_{\theta}(M))$ the cohomology groups of $M$ with coefficients in the sheaf $\Phi_{\theta}(M)$. Thus, we obtain a de Rham theorem for the  Lichnerowicz type cohomology attached to a function, that is
\begin{theorem}
The Lichnerowicz type cohomology groups attached to a function of $M$ are given by
\begin{equation}
H^{r}_{f,\theta}(M)\cong H^r(M,\Phi_{\theta}(M)).
\label{x4}
\end{equation}
\end{theorem}
\begin{remark}
Using the classical Poincar\'{e} Lemma for the de Rham operator $d$ and a similar argument as above we can also obtain that the cohomology attached to a function $H^\bullet_f(M)$ satisfies a de Rham isomorphism, that is $H^\bullet_f(M)\cong H^\bullet(M;\mathbb{R})$. 
\end{remark}

\subsection{Singular $r$-forms}
Let $S=f^{-1}(\{0\})$. According to \cite{Mo2} a form $\varphi\in\Omega^r(M\setminus S)$ is called a \textit{singular $r$-form} if $f^r\varphi$ can be extended to a smooth $r$-form on whole $M$. We denote the space of singular $r$-forms by $\Omega^r_f(M)$.

If $\varphi\in\Omega^r_f(M)$ is a singular $r$-form then $d\varphi$ is a singular $(r+1)$-form, see \cite{Mo2}. Then $d_\theta\varphi=d\varphi-\theta\wedge\varphi$ is a singular $r+1$-form. In fact we have
\begin{displaymath}
f^{r+1}d_{\theta}\varphi=d_{\theta}(f^{r+1}\varphi)-(r+1)df\wedge f^r\varphi,
\end{displaymath}
so $f^{r+1}d_{\theta}\varphi$ also extend to a smooth form on $M$. Therfore we obtain a chain complex $(\Omega^{\bullet}_f(M),d_{\theta})$ called the Lichnerowicz complex of singular forms. Similar to Proposition 2.4 from \cite{Mo2} we have
\begin{proposition}
The cohomology of $(\Omega^{\bullet}_{f}(M),d_{\theta})$ is isomorphic to $H^{\bullet}_{f,\theta}(M)$.
\end{proposition}
\begin{proof}
Define a map of chain complexes $\chi:(\Omega^{\bullet}_f(M),d_{\theta})\rightarrow (\Omega^{\bullet}(M),d_{f,\theta})$ by setting
\begin{displaymath}
\chi^r:\Omega^r_f(M)\rightarrow\Omega^r(M)\,,\,\chi^r(\varphi)=f^r\varphi.
\end{displaymath}
Then, using \eqref{x3}, we obtain 
\begin{equation}
d_{f,\theta}(\chi^r(\varphi))=\chi^{r+1}(d_{\theta}\varphi),
\label{0}
\end{equation}
hence $\chi$ induces an isomorphism between corresponding cohomologies.
\end{proof}

\subsection{Dependence on the function $f$}
As in the case of the cohomology $H^{\bullet}_f(M)$, a natural question to ask about the cohomology $H^{\bullet}_{f,\theta}(M)$ is how it depends on the function $f$. Similar with the Proposition 3.2. from \cite{Mo2},  we explain this fact for our cohomology. In fact we have
\begin{proposition}
If $h\in C^{\infty}(M)$ does not vanish, then cohomologies $H^{\bullet}_{f,\theta}(M)$ and $H^{\bullet}_{fh,\theta}(M)$ are isomorphic.
\end{proposition}
\begin{proof}
For each $r\in\mathbb{N}$, consider the linear isomorphism
\begin{equation}
\Phi^r:\Omega^{r}(M)\rightarrow\Omega^{r}(M)\,\,,\,\,\Phi^r(\varphi)=\frac{\varphi}{h^r}.
\label{II3}
\end{equation} 
If $\varphi\in\Omega^{r}(M)$, one checks easily that
\begin{equation}
\Phi^{r+1}(d_{fh,\theta}\varphi)=d_{f,\theta}(\Phi^r(\varphi)).
\label{II4}
\end{equation}
Indeed, we have 
\begin{eqnarray*}
\Phi^{r+1}(d_{fh,\theta}\varphi) &=& \Phi^{r+1}(d_{fh}\varphi-fh\theta\wedge\varphi)\\
&=& \Phi^{r+1}(d_{fh}\varphi)-\Phi^{r+1}(fh\theta\wedge\varphi)\\
&=& d_f(\Phi^r(\varphi))-f\theta\wedge\Phi^r(\varphi)\\
&=& d_{f,\theta}(\Phi^r(\varphi)),
\end{eqnarray*}
where we have used the relation $\Phi^{r+1}(d_{fh}\varphi)=d_{f}(\Phi^r(\varphi))$ from \cite{Mo2}.

Thus $\Phi$ induces an isomorphism between cohomologies $H^{\bullet}_{f,\theta}(M)$ and $H^{\bullet}_{fh,\theta}(M)$.
\end{proof}

\begin{corollary}
\label{depf}
If the function $f$ does not vanish, then $H^{\bullet}_{f,\theta}(M)$ is isomorphic to the Lichnerowicz cohomology $H^{\bullet}_{\theta}(M)$.
\end{corollary}
\begin{proof}
We take $h=\frac{1}{f}$ in the above proposition.
\end{proof}
We also have
\begin{corollary}
If $f$ and $g$ are smooth functions on $M$ such that $S=f^{-1}(0)=g^{-1}(0)$ and $f=g$ on some neighborhood of $S$, then $H^{\bullet}_{f,\theta}(M)\cong H^{\bullet}_{g,\theta}(M)$.
\end{corollary}

\subsection{Dependence on the class of $\theta$}
Another natural question to ask about the cohomology $H^{\bullet}_{f,\theta}(M)$ is if it depends on the class of $\theta$ as in the case of Lichnerowicz cohomology $H^{\bullet}_{\theta}(M)$. We have
\begin{proposition}
\label{deptheta}
The Lichnerowicz type cohomology attached to a function $f$ depends only on the class of $\theta$. In fact, we have the  isomorphism 
\begin{displaymath}
H^{\bullet}_{f,\theta-d\sigma}(M)\cong H^{\bullet}_{f,\theta}(M).
\end{displaymath}
\end{proposition}
\begin{proof}
By direct calculus we easily obtain $d_{f,\theta}(e^{\sigma}\varphi)=e^{\sigma}d_{f,\theta-d\sigma}\varphi$, where $\sigma$ is a smooth function and thus the map $[\varphi]\mapsto[e^{\sigma}\varphi]$ establishes an isomorphism between cohomologies  $H^{\bullet}_{f,\theta-d\sigma}(M)$ and $H^{\bullet}_{f,\theta}(M)$.
\end{proof}
\begin{remark}
The above result says that locally, the Lichnerowicz type cohomology attached to a function $f$ is isomorphic with the cohomology attached to a function after a change $\varphi\mapsto e^{-\sigma}\varphi$ with $\sigma$ a smooth function which satisfies $d\sigma=\theta$, that is $d_{f,\theta}$ is the unique differential in $\Omega^{\bullet}(M)$ which makes the multiplication by the smooth function $e^{-\sigma}$ an isomorphism of cochain complexes $e^{-\sigma}:(\Omega^{\bullet}(M),d_{f,\theta})\rightarrow(\Omega^{\bullet}(M),d_f)$. Consequently, we have $H^{\bullet}_{f,\theta}(M)\cong H^{\bullet}_{f}(M)$. 
\end{remark}
\begin{example}
Let $(M,\omega,\theta)$ be a locally conformal symplectic manifold, that is $\theta\in\Omega^1(M)$ is $d$-closed and $\omega\in\Omega^2(M)$ is $d_\theta$-closed, i.e. $d_\theta\omega=d\omega-\theta\wedge\omega=0$. Therefore, such a manifold admits two cohomological invariants: $[\theta]\in H^1_{dR}(M)$ and $[\omega]\in H^2_\theta(M)$. On the other hand, for  $f\in C^\infty(M)$, we have $d_{f}(f\theta)=0$ and $d_{f,\theta}(f^2\omega)=0$, which says that a locally conformal symplectic manifold $(M,\omega,\theta)$ admits another two cohomological invariants
\begin{equation}
\label{o1}
[f\theta]\in H^1_f(M)\,\,{\rm and}\,\,[f^2\omega]\in H^2_{f,\theta}(M).
\end{equation}
Moreover, if we consider a conformal equivalent locally conformal symplectic structure on $M$ defined by $\theta^\prime=\theta-df^2/f^2$ and $\omega^\prime=(1/f^2)\omega$, then we have $0=d_{f,\theta^\prime}(f^2\omega^\prime)=d_{f,\theta^\prime}\omega$. Thus, we have 
\begin{displaymath}
[\omega]\in H^2_{f,\theta^\prime}(M)\cong H^2_{f,\theta}(M)\cong H^2_\theta(M),
\end{displaymath}
where the first isomorphism is given by Proposition \ref{deptheta} and the second one follows by Corollary \ref{depf}.
\end{example}

\subsection{Relative cohomology}

The relative de Rham cohomology with respect to a smooth map between two manifolds was first defined in \cite{B-T} p. 78. Also, a relative vertical cohomology of foliated manifolds can be found for instance in \cite{Te}. In \cite{Mo2} is given a relative cohomology for $H^{\bullet}_f(M)$ and in \cite{I} is studied a relative cohomology for $H^{\bullet}_{\theta}(M)$. In this subsection we construct a similar version for our combined cohomology $H^{\bullet}_{f,\theta}(M)$.

Let $\mu:M\rightarrow M^{'}$ be a smooth map between two smooth manifolds. Taking into acount the standard relation  $d\mu^*=\mu^*d^{'}$, (here $d^{'}$ denotes the exterior derivative on $M^{'}$), we obtain
\begin{equation}
d_{\mu^*f}\mu^*=\mu^*d_{f}^{'}\,,\,f\in C^{\infty}(M^{'})\,,\,\mu^*f=f\circ \mu\in C^\infty(M).
\label{II5}
\end{equation}
Indeed, for $\varphi\in\Omega^{r}(M^{'})$, we have
\begin{eqnarray*}
d_{\mu^*f}(\mu^*\varphi) &=&\mu^*f d(\mu^*\varphi)-rd(\mu^*f)\wedge\mu^*\varphi\\
&=&\mu^*f\mu^*(d^{'}\varphi)-r\mu^*(d^{'}f)\wedge\mu^*\varphi\\
&=&\mu^*(fd^{'}\varphi)-\mu^*(rd^{'}f\wedge\varphi)\\
&=& \mu^*(d_{f}^{'}\varphi).
\end{eqnarray*}
The relation \eqref{II5} says that we have the homomorphism
\begin{displaymath}
\mu^*:H^{\bullet}_f(M^{'})\rightarrow H^{\bullet}_{\mu^*f}(M)\,,\,\mu^*[\varphi]=[\mu^*\varphi].
\end{displaymath}
Now, taking into account \eqref{II5} we obtain
\begin{equation}
d_{\mu^*f,\mu^*\theta}\mu^*=\mu^*d^{'}_{f,\theta}
\label{II6}
\end{equation}
for any smooth function $f\in C^{\infty}({M}^{'})$ and for any closed  $1$-form $\theta\in\Omega^1(M^{'})$. Indeed, for $\varphi\in\Omega^{r}(M^{'})$, we have
\begin{eqnarray*}
d_{\mu^*f,\mu^*\theta}(\mu^*\varphi) &=&d_{\mu^*f}(\mu^*\varphi)-\mu^*f\mu^*\theta\wedge\mu^*\varphi\\
&=&\mu^*d^{'}_{f}\varphi-\mu^*(f\theta\wedge\varphi)\\
&=& \mu^*(d_{f,\theta}^{'}\varphi).
\end{eqnarray*}
The relation \eqref{II6} says that we have the homomorphism
\begin{equation}
\mu^*:H^{\bullet}_{f,\theta}(M^{'})\rightarrow H^{\bullet}_{\mu^*f,\mu^*\theta}(M)\,,\,\mu^*[\varphi]=[\mu^*\varphi].
\label{1}
\end{equation}
\begin{remark}
If $\mu$ is a diffeomorphism then $H^{\bullet}_{f,\theta}(M^{'})\cong H^{\bullet}_{\mu^*f,\mu^*\theta}(M)$.
\end{remark}
Now, we define the differential complex $\left(\Omega^\bullet(\mu),\widetilde{d}_{f,\theta}\right)$, where
\begin{displaymath}
\Omega^{r}(\mu)=\Omega^{r}(M^{'})\oplus\Omega^{r-1}(M),\,\,{\rm and}\,\,\,\widetilde{d}_{f,\theta}(\varphi, \psi)=(-d_{f,\theta}^{'}\varphi, \mu^*\varphi+d_{\mu^*f,\mu^*\theta}\psi).
\end{displaymath}
Taking into account $d_{f,\theta}^{'2}=d_{\mu^*f,\mu^*\theta}^2=0$ and ({\ref{II6}}) we easily verify that $\widetilde{d}_{f,\theta}^2=0$. Denote the cohomology groups of this complex by $H^{\bullet}_{f,\theta}(\mu)$. They are called the \textit{Lichnerowicz cohomology groups attached to the function $f$ relative to the smooth map $\mu$}. 
\begin{example}
To obtain an example of such cohomological invariant, we consider a locally conformal K\"{a}hler manifold $(M,\omega,\theta)$ and  its K\"{a}hler covering map $\pi:(\widetilde{M},\widetilde{\omega})\rightarrow (M,\omega,\theta)$ for which, the pullback of $\theta$ is exact, that is $\pi^*\theta=d\sigma$ for some $\sigma\in C^{\infty}(\widetilde{M})$. According to \cite{O-V1, O-V2}, we have 
\begin{equation}
\pi^*\omega=d_{\pi^*\theta}(d^c\sigma),
\label{III6}
\end{equation}
for the real operator $d^c=i(\overline{\partial}-\partial)$, where $d=\partial+\overline{\partial}$ is the decomposition of the exterior derivative on the K\"{a}hler manifold $(\widetilde{M},\widetilde{\omega})$. Now, if we consider the relative Lichnerowicz cohomology attached to a function $f\in C^\infty(M)$ associated to the covering map $\pi:(\widetilde{M},\widetilde{\omega})\rightarrow (M,\omega,\theta)$, by direct calculus, we have  
\begin{displaymath}
\widetilde{d}_{f,\theta}(f^2\omega,-\pi^*f\cdot d^c\sigma)=\left(-d_{f,\theta}(f^2\omega),\pi^*(f^2\omega)-d_{\pi^*f,\pi^*\theta}(\pi^*f\cdot d^c\sigma)\right)=(0,0).
\end{displaymath}
Therefore, $(f^2\omega,-\pi^* f\cdot d^c\sigma)$ defines a cohomology class in  $H^2_{f,\theta}(\pi)$, which is called the \textit{Lichnerowicz class attached to $f$ relativ to the covering map of a lcK manifold}.
\end{example}
Now, if we regraduate the complex $\Omega^{r}(M)$ as $\widetilde{\Omega}^{r}(M)=\Omega^{r-1}(M)$, then we obtain an exact sequence of
differential complexes
\begin{equation}
0\longrightarrow(\widetilde{\Omega}^{r}(M),d_{\mu^*f,\mu^*\theta})\stackrel{\alpha}{\longrightarrow}(\Omega^{r}(\mu),\widetilde{d}_{f,\theta})\stackrel{\beta}{\longrightarrow}(\Omega^{r}(M^{'}),d^{'}_{f,\theta})\longrightarrow0
\label{II7}
\end{equation}
with the obvious mappings $\alpha$ and $\beta$ given by $\alpha(\psi)=(0, \psi)$ and $\beta(\varphi, \psi)=\varphi$, respectively. From ({\ref{II7}})
we have an exact sequence in cohomologies 
\begin{displaymath}
\ldots\longrightarrow H^{r-1}_{\mu^*f,\mu^*\theta}(M)\stackrel{\alpha^*}{\longrightarrow}H^{r}_{f,\theta}(\mu)\stackrel{\beta^*}{\longrightarrow}H^{r}_{f,\theta}(M^{'})\stackrel{\delta^*}{\longrightarrow}H^{r}_{\mu^*f,\mu^*\theta}(M)\longrightarrow\ldots .
\end{displaymath}
It is easily seen that $\delta^*=\mu^*$. Here $\mu^*$ denotes the corresponding map between cohomology groups. Let $\varphi\in\Omega^{r}(M^{'})$ be a $d^{'}_{f,\theta}$-closed form, and $(\varphi, \psi)\in \Omega^{r}(\mu)$. Then $\widetilde{d}_{f,\theta}(\varphi, \psi)=(0, \mu^*\varphi+d_{\mu^*f,\mu^*\theta}\psi)$ and by the definition of the operator $\delta^*$ we have
\begin{displaymath}
\delta^*[\varphi]=[\mu^*\varphi+d_{\mu^*f,\mu^*\theta}\psi]=[\mu^*\varphi]=\mu^*[\varphi].
\end{displaymath}
Hence, we get a long exact sequence
\begin{equation}
\ldots\longrightarrow H^{r-1}_{\mu^*f,\mu^*\theta}(M)\stackrel{\alpha^*}{\longrightarrow}H^{r}_{f,\theta}(\mu)\stackrel{\beta^*}{\longrightarrow}H^{r}_{f,\theta}(M^{'})\stackrel{\mu^*}{\longrightarrow}H^{r}_{\mu^*f,\mu^*\theta}(M)\longrightarrow\ldots ,
\label{II8}
\end{equation}
which implies
\begin{proposition}
If the  manifolds $M$ and $M^{'}$ are of the $n$-th and $n^{'}$-th dimension, respectively, then
\begin{enumerate}
\item[(i)] $\beta^*:H^{n+1}_{f,\theta}(\mu)\rightarrow H^{n+1}_{f,\theta}(M^{'})$ is an epimorphism,
\item[(ii)] $\alpha^*:H^{n^{'}}_{\mu^*f,\mu^*\theta}(M)\rightarrow H^{n^{'}+1}_{f,\theta}(\mu)$  is an epimorphism, 
\item[(iii)] $\beta^*:H^{r}_{f,\theta}(\mu)\rightarrow H^{r}_{f,\theta}(M^{'})$ is an isomorphism for $r>n+1$,
\item[(iv)] $\alpha^*:H^{r}_{\mu^*f,\mu^*\theta}(M)\rightarrow H^{r+1}_{f,\theta}(\mu)$ is an isomorphism for $r>n^{'}$,
\item[(v)] $H^{r}_{f,\theta}(\mu)=0$ for $r> {\rm max}\{n+1,n^{'}\}$.
\end{enumerate}
\end{proposition}

\subsection{A Mayer-Vietoris sequence}

Since the differentials $d_{f,\theta}$ commutes with the restrictions to open subsets, one can construct  a Mayer-Vietoris exact sequence, in the same way as for the de Rham cohomology, see \cite{B-T, Va2} (or Lichnerowicz cohomology see \cite{H-R}, or cohomology attached to a function see \cite{Mo2}), namely:

Suppose $M$ is the union of two open subsets $U, V$ . Then the following is a short
exact sequence of cochain complexes
\begin{displaymath}
0\rightarrow(\Omega^{\bullet}(M),d_{f,\theta})\stackrel{\alpha}{\rightarrow}(\Omega^{\bullet}(U)\oplus\Omega^{\bullet}(V),d_{f|_U,\theta|_U}\oplus d_{f|_V,\theta|_V})\stackrel{\beta}{\rightarrow}
\end{displaymath}
\begin{displaymath}
\stackrel{\beta}{\rightarrow}(\Omega^{\bullet}(U\cap V),d_{f|_{U\cap V},\theta_{U\cap V}})\rightarrow 0
\end{displaymath}
where $\alpha(\varphi)=(\varphi|_U,\varphi|_V)$ and $\beta(\varphi,\psi)=\varphi|_{U\cap V}-\psi|_{U\cap V}$. So we obtain the following Mayer-Vietoris sequence:
\begin{theorem}
If $\mathcal{U}=\{U, V\}$ is an open cover of $M$, we have the long exact sequence
\begin{equation}
\ldots\rightarrow H^{r}_{f,\theta}(M)\stackrel{\alpha_*}{\rightarrow} H^{r}_{f|_U,\theta|_U}(U)\oplus H^r_{f|_V,\theta|_V}(V)\stackrel{\beta_*}{\rightarrow}
\label{2}
\end{equation} 
\begin{displaymath}
\stackrel{\beta_*}{\rightarrow}H^{r}_{f|_{U\cap V},\theta|_{U\cap V}}(U\cap V)\stackrel{\delta}{\rightarrow}H^{r+1}_{f,\theta}(M)\rightarrow\ldots
\end{displaymath}
where 
$\alpha_*([\varphi])=([\varphi_U],[\varphi_V])\,,\,\beta_*([\varphi],[\psi])=[\varphi|_{U\cap V}-\psi|_{U\cap V}]$, $\delta([\sigma])=[d_{f}\lambda|_V\wedge\sigma]=-[d_{f}\lambda|_U\wedge\sigma]$. Here $\{\lambda_U,\lambda_V\}$ is a partition of unity
subordinate to $\{U,V\}$ and the forms under consideration are assumed to be extended by $0$ to the whole $M$.
\end{theorem}
\begin{example}
Consider $M=\mathbb{R}^2-\{(-1,0),(1,0)\}$ and let $\theta$ and $\eta$ be a generator of $H^1_{dR}(M)$ supported in $(-\infty,0)\times \mathbb{R}$ and $U:=(0,\infty)\times\mathbb{R}$, respectively. Then taking into account that $d_{\theta}\eta=0$ and the fact that $\eta|_U$ cannot be $d_{\theta|_U}$-exact, see \cite{H-R}, we easily obtain that $d_{f,\theta}(f\eta)=0$ and $f|_U\eta|_U$ cannot be $d_{f|U,\theta|_U}$-exact, for a smooth function on $M$. Using Mayer-Vietoris sequence for the cohomology $H^{\bullet}_{f,\theta}$ from \eqref{2} one can show that $f\eta$ generates $H^1_{f,\theta}(M)$
\end{example}

\subsection{Homotopy invariance}

\begin{definition}
(\cite{Mo2}). Let $M$ and $M^{'}$ two smooth manifolds and $f\in C^{\infty}(M)$ and $f^{'}\in C^{\infty}(M^{'})$.
A \textit{morphism} from the pair $(M, f)$ to the pair $(M^{'}, f^{'})$ is a pair $(\phi, \alpha)$ formed by a morphism (smooth map) $\phi:M\rightarrow M^{'}$ and a real valued function $\alpha:M\rightarrow \mathbb{R}$, such that $\alpha$ does not vanish on $M$ and $\phi^*f^{'}=f^{'}\circ \phi=\alpha f$. 
\end{definition}
We will say that the pairs $(M,f)$ and $(M^{'},f^{'})$ are \textit{equivalent} if there exists a morphism $\Phi=(\phi,\alpha)$ between these two pairs where $\phi$ is a diffeomorphism. This notion of equivalence between the pairs is sometimes called "contact equivalence" in singularity theory. In \cite{Mo2} is proved that a morphism $\Phi=(\phi,\alpha)$ from the pair $(M,f)$ to the pair $(M^{'},f^{'})$ induces a chain map $\Phi^*:\left(\Omega^{\bullet}(M^{'}),d_{f^{'}}\right)\rightarrow\left(\Omega^{\bullet}(M),d_{f}\right)$ defined by
\begin{displaymath}
\Phi^*:\Omega^{r}(M^{'})\rightarrow\Omega^{r}(M)\,,\,\Phi^*(\varphi)=\frac{\phi^*\varphi}{\alpha^r},
\end{displaymath}
and this map  induces an homomorphism in cohomology, that is $\Phi^*:H^{\bullet}_{f^{'}}(M^{'})\rightarrow H^{\bullet}_f(M)$.  If $\Phi$ is diffeomorphism then $H^{r}_{f^{'}}(M^{'})$ and $H^{r}_{f}(M)$ are isomorphic. 

Now, taking into account that for any $\varphi\in\Omega^r(M^{'})$ we have $\Phi^*(d_{f^{'}}\varphi)=d_f(\Phi^*(\varphi))$, see \cite{Mo2}, by direct calculus we obtain
\begin{equation}
\Phi^*(d_{f^{'},\theta}\varphi)=d_{f,\phi^*\theta}(\Phi^*(\varphi))
\label{II9}
\end{equation}
for any closed $1$-form $\theta$ on $M^{'}$ and $\varphi\in\Omega^r(M^{'})$.

Thus $\Phi$ induces an homomorphism in Lichnerowicz type cohomology attached to a function $\Phi^*:H^{\bullet}_{f^{'},\theta}(M^{'})\rightarrow H^{\bullet}_{f,\phi^*\theta}(M)$.  Moreover, if $\Phi$ is diffeomorphism then $H^{r}_{f^{'},\theta}(M^{'})$ and $H^{r}_{f,\phi^*\theta}(M)$ are isomorphic.
\begin{remark}
For $\alpha=1$ then we obtain the homomorphism from \eqref{1}.
\end{remark}
\begin{definition}
(\cite{Mo2}). A \textit{homotopy} from the pair $(M,f)$ to the pair $(M^{'},f^{'})$ is given by two smooth maps
\begin{displaymath}
h:M\times[0,1]\rightarrow M^{'}\,,\,a:M\times[0,1]\rightarrow\mathbb{R},
\end{displaymath}
such that for each $t\in[0,1]$, we have a morphism
\begin{displaymath}
H_t\equiv(h_t,a_t):(M,f)\rightarrow(M^{'},f^{'})
\end{displaymath}
(i.e., $a$ does not vanish, $(f^{'}\circ h)(x,t)=a(x,t)f(x)$), where $h_t=h(\cdot,t),\,a_t=a(\cdot,t)$.
\end{definition}
Now, if $H=(h,a)$ is a homotopy from $(M,f)$ to $(M^{'},f^{'})$, from above discussion we obtain a map at cohomology level
\begin{displaymath}
H_t^*:H^{\bullet}_{f^{'},\theta}(M^{'})\rightarrow H^{\bullet}_{f,h^*_t\theta}(M).
\end{displaymath}
For the Lichnerowicz cohomology $H^{\bullet}_{\theta}(M)$ the problem of homotopy invariance is solved by Lemma 1.1 from \cite{H-R}. For the cohomology attached to a function $H^{\bullet}_f(M)$ the problem of homotopy invariance is the following: given a homotopy $H$, from
$(M,f)$ to $(M^{'},f^{'})$, is it true that $H^*_0=H^*_1$ at the cohomology level?  This problem is partial solved in \cite{Mo2} namely:  If the complements of the zero level sets of $f$ and $f^{'}$ are dense sets, then in degree zero we do have $H^*_0=H^*_1:H^0_f(M)\rightarrow H^0_{f^{'}}(M^{'})$. For higher degree, a partial result in the regular case is also given in \cite{Mo2}. For our Lichnerowicz type cohomology attached to a function $H^{\bullet}_{f,\theta}$ this problem is still open, but in the next section we prove a homotopy invariance in the regular case.

\subsection{K\"{u}nneth type formula}
The purpose of K\"{u}nneth formula is the computation of the cohomology of the Cartesian product when we known the cohomologies of the factors.

Suppose we have two manifolds $M_1$, $M_2$ and two closed $1$-forms $\theta_1$ and $\theta_2$, on $M_1$ and $M_2$, respectively. Let $\theta:={\rm pr}_1^*\theta_1+{\rm pr}_2^*\theta_2\in\Omega^1(M_1\times M_2)$ which is also closed. Then one defines a mapping
\begin{displaymath}
\Psi:\Omega^k(M_1)\times\Omega^l(M_2)\rightarrow\Omega^{k+l}(M_1\times M_2)\,,\,\Psi(\varphi,\psi)={\rm pr}_1^*\varphi\wedge{\rm pr}_2^*\psi.
\end{displaymath}
Then if we consider two smooth functions $f_1$ and $f_2$ on $M_1$ and $M_2$, respectively such that $f:={\rm pr}_1^*f_1={\rm pr}_2^*f_2$ then by direct computation, we obtain 
\begin{displaymath}
d_{f,\theta}(\Psi(\varphi,\psi))=\Psi(d_{f_1,\theta_1}\varphi,\psi)+(-1)^{\deg \varphi}\Psi(\varphi,d_{f_2,\theta_2}\psi)
\end{displaymath}
and, hence we have an induced mapping
\begin{equation}
\label{k}
H^{\bullet}_{f_1,\theta_1}(M_1)\otimes H^{\bullet}_{f_2,\theta_2}(M_2)\rightarrow H^{\bullet}_{f,\theta}(M_1\times M_2).
\end{equation}

According to \cite{B-T}, a covering $\mathcal{U}$ of a manifold $M$ is said to be \textit{good}, if for all $n\in\mathbb{N}$ and $U_1,\ldots,U_n\in\mathcal{U}$ the intersection $U_1\cap U_2\ldots\cap U_n$ is either empty or contractible. 

We have the following K\"{u}nneth type formula for our cohomology.
\begin{theorem}
Suppose that $M_1$ and $M_2$ have good covers and let $\theta_1$ a closed $1$-form on $M_1$ and $\theta_2$ a closed $1$-form on $M_2$, respectively. Also, we consider two smooth functions $f_1$ and $f_2$ on $M_1$ and $M_2$, respectively such that $f:={\rm pr}_1^*f_1={\rm pr}_2^*f_2$. Then, the map from \eqref{k} is an isomorphism.
\end{theorem}

\section{The regular case}
\setcounter{equation}{0}
In this section we study the \textit{regular} case i.e., the case where the function $f$ does not have singularities in a neighborhood of its zero set (i.e. $0$ is a regular value). The subset $S =f^{-1}(\{0\})$ is then an embedded submanifold of $M$. We also assume that $S$ is connected. In this case, the cohomology attached to a function $H^{\bullet}_f(M)$ is related with the de Rham cohomologies $H^{\bullet}_{dR}(M)$ and $H^{\bullet-1}_{dR}(S)$, see \cite{Mo2}. Similarly, we can relate in this case the Lichnerowicz type cohomology attached to a function $H^{\bullet}_{f,\theta}(M)$ with the Lichnerowicz cohomologies $H^{\bullet}_{\theta}(M)$ and $H^{\bullet-1}_{i^*\theta}(S)$, where $i:S\rightarrow M$ is the natural inclusion.  Also in this regular case, we obtain a homotopy invariance.
\begin{theorem}
\label{thmain}
If $0$ is a regular value of $f$ then, for each $r\geq1$, there is an isomorphism
\begin{equation}
H^r_{f,\theta}(M)\cong H^r_{\theta}(M)\oplus H^{r-1}_{i^*\theta}(S),
\label{r1}
\end{equation}
for any closed $1$-form $\theta$ on $M$ without singularities.
\end{theorem}
\begin{proof}
The proof follows in a similar manner with the proof of Theorem 4.1 from \cite{Mo2} and we need to briefly recall some preliminary results.

Let $U\subset U^{'}$ be tubular neighborhoods of $S$. We may assume that $U=S\times]-\varepsilon,\varepsilon[$ and $U^{'}=S\times]-\varepsilon^{'},\varepsilon^{'}[$, with $\varepsilon^{'}>\varepsilon$, and that 
\begin{displaymath}
f|_{U^{'}}:S\times]-\varepsilon^{'},\varepsilon^{'}[\rightarrow\mathbb{R}\,,\,(x,t)\mapsto t.
\end{displaymath}
Let us consider the projection $\pi:U^{'}\rightarrow S$ and $\rho:\mathbb{R}\rightarrow\mathbb{R}$ be a smooth function which is $1$ on $[-\varepsilon,\varepsilon]$ and has support contained in $[-\varepsilon^{'},\varepsilon^{'}]$. Note that the function $\rho\circ f$ is $1$ on $U$, and we can assume that the function $\rho\circ f$ vanishes on $M\setminus U^{'}$. 

If $\psi$ is a form on $S$, we will denote by $\overline{\psi}$ the form $\rho(f)\pi^*\psi$. Notice that from
\begin{displaymath}
d\overline{\psi}=\rho(f)\pi^*(d\psi)+\rho^{'}(f)df\wedge\pi^*\psi
\end{displaymath}
it easily follows that
\begin{equation}
df\wedge d\overline{\psi}=df\wedge\overline{d\psi}.
\label{r2}
\end{equation}
Now we notice that for the closed $1$-form $\theta$ on $M$ we have
\begin{equation}
\theta|_{U^{'}}\wedge\overline{\psi}=(i\circ\pi)^*\theta\wedge\overline{\psi}=\rho(f)\pi^*(i^*\theta)\wedge\pi^*\psi=\rho(f)\pi^*(i^*\theta\wedge\psi)=\overline{i^*\theta\wedge\psi},
\label{r3}
\end{equation}
for any form $\psi$ on $S$.

In the sequel we denote by $\zeta$ the linear application
\begin{equation}
\zeta:\Omega^r(M)\oplus\Omega^{r-1}(S)\rightarrow \Omega^r(M)\,,\,\zeta(\varphi,\psi)=f^r\varphi+f^{r-1}df\wedge\overline{\psi}.
\label{r4}
\end{equation}
If $(\varphi,\psi)\in\Omega^r(M)\oplus\Omega^{r-1}(S)$, with $d_{\theta}\varphi=0$ and $d_{i^*\theta}\psi=0$, then using \eqref{r2} and \eqref{r3}, we find
\begin{eqnarray*}
d_{f,\theta}(\zeta(\varphi,\psi)) &=& f^{r+1}d_{\theta}\varphi-f^rdf\wedge(d\overline{\psi}-\theta|_{U^{'}}\wedge\overline{\psi})\\
&=& f^{r+1}d_{\theta}\varphi-f^rdf\wedge\overline{d_{i^*\theta}\psi}=0.
\end{eqnarray*}
Similarly, one checks that if $\varphi\in\Omega^{r-1}(M)$ and $\psi\in\Omega^{r-2}(S)$, then
\begin{displaymath}
\zeta(d_{\theta}\varphi,d_{i^*\theta}\psi)=d_{f,\theta}(f^{r-1}\varphi-f^{r-2}df\wedge\overline{\psi}).
\end{displaymath}
We conclude that $\zeta$ induces a map at the level of cohomology
\begin{equation}
\zeta^*:H^{\bullet}_{\theta}(M)\oplus H^{\bullet-1}_{i^*\theta}(S)\rightarrow H^{\bullet}_{f,\theta}(M)\,,\,\zeta^*([\varphi],[\psi])=[\zeta(\varphi,\psi)].
\label{r5}
\end{equation}
Finally, according to \cite{Mo2}, $\zeta$ is bijective for all $r\geq1$ and so the theorem follows.
\end{proof}
\begin{example}
Let $S^1=\{(x_1,x_2)\in\mathbb{R}^2\,|\,x_1^2+x_2^2=1\}$ be the $1$-sphere and $f:S^1\rightarrow\mathbb{R}$ the function $f(x_1,x_2)=x_1$, so that $S=f^{-1}(\{0\})$ is the equator. Then taking into account that $H^0_{\theta}(S^1)=0$, for any closed, non-exact $1$-form $\theta$ on $S^1$, see Example 1.6 from \cite{H-R}, we obtain $H^0_{f,\theta}(S^1)=0$.
\end{example}
In the regular case we have the following homotopy invariance:
\begin{proposition}
Let $U$ and $W$ be tubular neighborhoods of $S_f =f^{-1}(0)$ and $S_g=g^{-1}(0)$, respectively.
We assume that $f$ and $g$ do not have singularities on $U$ and $W$. If $H_t$ is a homotopy from $(U, f )$ to $(W, g)$, then the induced linear applications between the cohomology spaces are the same: $H_1^*=H_0^*$.
\end{proposition}
\begin{proof}
We can assume that $U=S_f\times]-\varepsilon,\varepsilon[$ and $W=S_g\times]-\varepsilon,\varepsilon[$, with 
\begin{displaymath}
f(x,\rho)=\rho\,,\,g(y,\tau)=\tau.
\end{displaymath}
We denote by $\Psi_f$ and $\Psi_g$ the linear maps:
\begin{displaymath}
\Psi_f:H^r_{h_t^*(\theta|_W)}(U)\oplus H^{r-1}_{h_t^*(\theta|_W)}(U)\rightarrow H^r_{f,h_t^*(\theta|_W)}(U)\,,\,\Psi_f([\varphi],[\psi])=[\rho^r\varphi+\rho^{r-1}d\rho\wedge\psi],
\end{displaymath}
\begin{displaymath}
\Psi_g:H^r_{\theta|_W}(W)\oplus H^{r-1}_{\theta|_W}(W)\rightarrow H^r_{g,\theta|_W}(W)\,,\,\Psi_g([\varphi],[\psi])=[\tau^r\varphi+\tau^{r-1}d\tau\wedge\psi],
\end{displaymath}
which by the previous theorem, are isomorphisms.

Now, we set $K^*_t=\Psi_f^{-1}\circ H_t^*\circ\Psi_g$, for every $t\in[0,1]$. If $([\varphi],[\psi])\in H^r_{\theta|_W}(W)\oplus H^{r-1}_{\theta|_W}(W)$, then by a similar calculus as in the proof of Proposition 4.12 from \cite{Mo2} we have
\begin{displaymath}
H_t^*(\Psi_g([\varphi],[\psi]))=\left[\rho^rh_t^*\varphi+\rho^rd(\log |a_t|)\wedge h_t^*\psi+\rho^{r-1}d\rho\wedge h_t^*\psi\right].
\end{displaymath}
Now taking into account $d_{h_t^*(\theta|_W)}h_t^*\psi=h_t^*(d_{\theta|_W}\psi)=0$, we conclude that 
\begin{eqnarray*}
K_t^*([\varphi],[\psi]) &=& \left([h_t^*\varphi+d_{h_t^*(\theta|_W)}(\log |a_t|h_t^*\psi)], [h_t^*\psi]\right)\\
&=& \left([h_t^*\varphi], [h_t^*\psi]\right).
\end{eqnarray*}
Since the Lichnerowicz cohomology is homotopy invariant, see Lemma 1.1 from \cite{H-R}, we have $K_1^*=K_0^*$ and it follows that $H_1^*=H_0^*$. 
\end{proof}
Finally, following step by step the proof of Proposition 4.13 from \cite{Mo2}, we obtain
\begin{proposition}
Let $H_t$ be a homotopy from $(M,f)$ to $(N,g)$. We suppose that $f$ and $g$ do not have singularities on tubular neighborhoods of $S_f$ and $S_g$, respectively. If $H^{r-1}_{(h_0^*\theta)|_S}(S)$ is trivial, then $H_0^*=H_1^*:H^r_{g,\theta}(N)\rightarrow H^r_{f,h_0^*\theta}(M)$, for every closed $1$-form $\theta$ on $N$. 
\end{proposition}

\section*{Acknowledgment}
The author cordially thanks to Referee(s) for useful comments and suggestions about the initial submission which improve substantially the presentation and the contents of this paper.

\noindent
Cristian Ida\\
Department of Mathematics and Computer Science\\
University Transilvania of Bra\c{s}ov\\
Address: Bra\c{s}ov 500091, Str. Iuliu Maniu 50, Rom\^{a}nia\\
email: \textit{cristian.ida@unitbv.ro}

\smallskip

\end{document}